\documentclass[12pt]{amsart}

\usepackage{amsmath,amssymb,amsfonts,fullpage,amssymb,amsthm}
\usepackage{verbatim}
\usepackage[usenames]{color}
\usepackage{hyperref}  
\usepackage{url}

\newtheorem{thm}{Theorem}[section]
\newtheorem{lem}[thm]{Lemma}
\newtheorem{cor}[thm]{Corollary} 

\newtheorem*{claim}{Claim}

\newtheorem*{fact1}{Fact 1}
\newtheorem*{fact2}{Fact 2}
\newtheorem*{fact3}{Fact 3}

\theoremstyle{remark}
\newtheorem{rem}{Remark}[section]

\theoremstyle{definition}
\newtheorem{defn}[thm]{Definition}

\newtheorem{coro}{Corollary}[section]
\theoremstyle{definition}
\newtheorem*{nota}{Notation}

\theoremstyle{remark}

\newcommand{\fin}{\mathrm{fin}}

\newcommand{\om}{\omega}

\newcommand{\sse}{\subseteq}
\newcommand{\contains}{\supseteq}

\DeclareMathOperator{\ran}{ran}
\DeclareMathOperator{\dom}{dom}

\newcommand{\bN}{\mathbb{N}}

\DeclareMathOperator{\depth}{depth}
\DeclareMathOperator{\supp}{supp}

\DeclareMathOperator{\FIN}{FIN}

\newcommand{\ra}{\rightarrow}

\newcommand{\lgl}{\langle}
\newcommand{\rgl}{\rangle}

\newcommand{\noprint}[1]{\relax}

\title{Topological Ramsey spaces and metrically Baire sets}
\author{Natasha Dobrinen and Jos\'e G. Mijares}
\address{Department of Mathematics, University of Denver, 2280 S Vine St., Denver, CO 80208, USA}
\email{natasha.dobrinen@du.edu\\ URL: http://web.cs.du.edu/$\sim$ndobrine}
\thanks{Dobrinen  was partially supported by  National Science Foundation Grant DMS-1301665}

\address{Department of Mathematics, University of Denver. 2280  S Vine St., Denver, CO 80208, USA}
\email{Jose.MijaresPalacios@du.edu}

\begin{document}

\maketitle

\begin{abstract}
We characterize a class of topological Ramsey spaces
such that each element $\mathcal R$ of the class induces a collection $\{\mathcal R_k\}_{k<\omega}$ of projected spaces which have the property that every Baire set is Ramsey.
Every projected space $\mathcal R_k$
is a subspace of the  corresponding space of length-$k$ approximation sequences with the Tychonoff, equivalently metric, topology.
 This answers a question of S. Todorcevic and generalizes the results of Carlson \cite{Carlson}, Carlson-Simpson \cite{CarSim2}, Pr\"omel-Voigt \cite{PromVoi},  and Voigt \cite{Voigt}.
We also present a new family of topological Ramsey spaces contained in the aforementioned class which generalize the spaces of ascending parameter words  of  Carlson-Simpson \cite{CarSim2} and  Pr\"omel-Voigt \cite{PromVoi}  and  the spaces $\FIN_m^{[\infty]}$, $0<m<\omega$, of block sequences defined by Todorcevic \cite{Todo}. 
\end{abstract}

\section{Introduction}

There exists a class of topological Ramsey spaces whose members admit, for every $k<\omega$, a set of sequences of $k$-approximations that can be understood as a topological space where every Baire set satisfies the Ramsey property. Such topological spaces inherit the Tychonoff, (equivalently, metric) topology from the space of approximation sequences associated to the elements of the underlying topological Ramsey space. 
We shall say that these  projected spaces have the property that 
every metrically Baire subset is Ramsey. 

The infinite Dual Ramsey Theorem of Carlson and Simpson  in \cite{CarSim2}
 was the first result where this phenomenon was seen. 
In that paper, it was shown that for
 Carlson-Simpson's space of equivalence relations on $\mathbb N$ with infinitely many equivalence classes, 
for each $k<\om$,
the projected space of equivalence relations on $\mathbb N$ with exactly $k$ equivalence classes
has the property that every metrically Baire set is Ramsey.
Other examples where this phenomenon occurs are 
 Pr\"omel-Voigt's spaces of parameter words, ascending parameter words and partial $G$-partitions (where $G$ is a finite group) \cite{PromVoi}; and Carlson's space of infinite dimensional vector subspaces of $F^{\mathbb N}$ (where $F$ is a finite field) \cite{Carlson}, in connection with an extension of the Graham-Leeb-Rothschild Theorem \cite{gralero} due to Voigt \cite{Voigt}. In this work, answering a question of Todorcevic, we give a characterization of this class of topological Ramsey spaces.

The theory of toplogical Ramsey spaces has experienced increasing developent in recent years. In the book \textit{Introduction to Ramsey spaces} \cite{Todo} by S. Todorcevic, most of the foundational results, examples and applications are presented within the framework of a more general type of Ramsey space (not neccesarily topological). A topological Ramsey space is the main object of the \textit{topological Ramsey theory} (see Section \ref{trs} for the definitions). Carlson and Simpson gave the first abstract exposition of this theory in \cite{CarSim}. The first known example of a topological Ramsey space was given in \cite{Ellen}, building upon earlier results like \cite{NashW,Galvin,GalPri}.  Recent developments in the study of topological Ramsey spaces have been made regarding connections to forcing, the theory of ultrafilters, selectivity, Tukey reducibility, parametrized partition theorems, canonization theorems, topological dynamics, structural Ramsey theory, Fra\"iss\'e classes and random objects, among others (see for instance \cite{DobTod0,DobTod,DobTod2,DMT,Mijares,Mijares2,Mijares3,MijaNie,MijaPad,RaghaTod,Tim1,Tim2}). 

In this paper we study a feature of some topological Ramsey spaces which had not been fully understood in the abstract setting. More precisely, we establish conditions of sufficiency which characterize those topological Ramsey spaces $\mathcal R$ with family of approximations $\mathcal{AR}= \bigcup_{k<\omega}\mathcal{AR}_k$ for which there exist topological spaces $\mathcal R_k\subseteq(\mathcal{AR}_k)^{\mathbb N}$, $k<\omega$, such that every Baire subset of $\mathcal R_k$ is Ramsey. Each $\mathcal R_k$ inherits the Tychonoff topology from $(\mathcal{AR}_k)^{\mathbb N}$, which results when $\mathcal{AR}_k$ is understood as a discrete space.

In Section \ref{trs} we give a brief description of the theory of topological Ramsey spaces. Section \ref{setting} contains the main result of the paper, the characterization announced in the previous paragraph. In Section \ref{general parameter}, 
 we introduce a class of topological Ramsey spaces which generalizes the spaces of ascending parameter words studied by Carlson-Simpson in \cite{CarSim2} and  Pr\"omel-Voigt in \cite{PromVoi} (see subsection \ref{ascending}), and which turns out to be also a generalization of the spaces $\FIN_m^{[\infty]}$, $0<m<\omega$, of block sequences defined by Todorcevic  in \cite{Todo}. We show that each space in this class admits projection spaces where every Baire set is Ramsey, fitting into the abstract setting introduced in Section \ref{setting}. In Section \ref{classical}, we show that the classical examples originally introduced in \cite{CarSim2,PromVoi,Todo,Voigt} fit the abstract setting given in Section \ref{setting}. These classical examples motivated this research. At the end of this article we comment about open questions related to the results in Sections \ref{setting} and \ref{general parameter}.

\section{Topological Ramsey spaces}\label{trs}

The four axioms which guarantee that a space is a topological Ramsey space (see Definition \ref{defn.5.3} below) can be found at the beginning of Chapter 5 of 
\cite{Todo}, which we reproduce in this Section.

Consider a triple
$(\mathcal{R},\le,r)$
of objects with the following properties.
$\mathcal{R}$ is a nonempty set,
$\le$ is a quasi-ordering on $\mathcal{R}$ and $r:\mathcal{R}\times\om\ra\mathcal{AR}$ is a mapping giving us the sequence \linebreak $(r_n(\cdot)=r(\cdot,n))$ of approximation mappings, where
$\mathcal{AR}$ is  the collection of all finite approximations to members of $\mathcal{R}$.  For every $B\in\mathcal R$, let

\begin{equation}
\mathcal R|B = \{A\in\mathcal R : A\leq B\}.
\end{equation}

For $a\in\mathcal{AR}$ and $B\in\mathcal{R}$, let
\begin{equation}
[a,B]=\{A\in\mathcal{R}:A\le B\mathrm{\ and\ }(\exists n)\ r_n(A)=a\}.
\end{equation}

For $a\in\mathcal{AR}$, let $|a|$ denote the length of the sequence $a$, that is, $|a|$ equals the integer $n$ for which $a=r_n(A)$, for some $A\in\mathcal R$.  If $m<n$, $a=r_m(A)$ and $b=r_n(A)$ then we will write $a=r_m(b)$. In particular, $a=r_m(a)$, and this is equivalent to $|a|=m$.  For $a,b\in\mathcal{AR}$, $a\sqsubseteq b$ if and only if $a=r_m(b)$ for some $m\le |b|$. $a\sqsubset b$ if and only if $a=r_m(b)$ for some $m<|b|$.
For each $n<\om$, $\mathcal{AR}_n=\{r_n(A):A\in\mathcal{R}\}$.
\vskip.1in

\begin{enumerate}
\item[\bf A.1]\rm
\begin{enumerate}
\item
$r_0(A)=\emptyset$ for all $A\in\mathcal{R}$.\vskip.05in
\item
$A\ne B$ implies $r_n(A)\ne r_n(B)$ for some $n$.\vskip.05in
\item
$r_n(A)=r_m(B)$ implies $n=m$ and $r_k(A)=r_k(B)$ for all $k<n$.\vskip.1in
\end{enumerate}
\item[\bf A.2]\rm
There is a quasi-ordering $\le_{\mathrm{fin}}$ on $\mathcal{AR}$ such that\vskip.05in
\begin{enumerate}
\item
$\{a\in\mathcal{AR}:a\le_{\mathrm{fin}} b\}$ is finite for all $b\in\mathcal{AR}$,\vskip.05in
\item
$A\le B$ iff $(\forall n)(\exists m)\ r_n(A)\le_{\mathrm{fin}} r_m(B)$,\vskip.05in
\item
$\forall a,b,c\in\mathcal{AR}\ [a\sqsubset b\wedge b\le_{\mathrm{fin}} c\ra\exists d\sqsubset c\ a\le_{\mathrm{fin}} d]$.\vskip.1in
\end{enumerate}
\end{enumerate}

$\depth_B(a)$ is the least $n$, if it exists, such that $a\le_{\mathrm{fin}}r_n(B)$.
If such an $n$ does not exist, then we write $\depth_B(a)=\infty$.
If $\depth_B(a)=n<\infty$, then $[\depth_B(a),B]$ denotes $[r_n(B),B]$.

\begin{enumerate}
\item[\bf A.3] \rm
\begin{enumerate}
\item
If $\depth_B(a)<\infty$ then $[a,A]\ne\emptyset$ for all $A\in[\depth_B(a),B]$.\vskip.05in
\item
$A\le B$ and $[a,A]\ne\emptyset$ imply that there is $A'\in[\depth_B(a),B]$ such that $\emptyset\ne[a,A']\sse[a,A]$.\vskip.1in
\end{enumerate}
\end{enumerate}

If $n>|a|$, then  $r_n[a,A]$ denotes the set $\{r_n(B) : B\in [a,A]\}$. Notice that $a\sqsubset b$ for every $b\in r_n[a,A]$.

\begin{enumerate}
\item[\bf A.4]\rm
If $\depth_B(a)<\infty$ and $\mathcal{O}\sse\mathcal{AR}_{|a|+1}$,
then there is $A\in[\depth_B(a),B]$ such that
$r_{|a|+1}[a,A]\sse\mathcal{O}$ or $r_{|a|+1}[a,A]\sse\mathcal{O}^c$.\vskip.1in
\end{enumerate}

The family $\{[a,B]:a\in\mathcal{AR},B\in\mathcal R\}$ forms a basis for the {\em Ellentuck} topology on $\mathcal{R}$;
it extends the usual metrizable topology on $\mathcal{R}$ when we consider $\mathcal{R}$ as a subspace of the Tychonoff cube $(\mathcal{AR})^{\bN}$.
Given the Ellentuck topology on $\mathcal{R}$,
the notions of nowhere dense, and hence of meager are defined in the usual way.
Thus, we may say that a subset $\mathcal{X}$ of $\mathcal{R}$ has the {\em property of Baire} if and only if $\mathcal{X}=\mathcal{O}\cap\mathcal{M}$ for some Ellentuck open set $\mathcal{O}\sse\mathcal{R}$ and some Ellentuck meager set $\mathcal{M}\sse\mathcal{R}$.

\begin{defn}[\cite{Todo}]\label{defn.5.2}
A subset $\mathcal{X}$ of $\mathcal{R}$ is {\em Ramsey} if for every $\emptyset\ne[a,A]$,
there is a $B\in[a,A]$ such that $[a,B]\sse\mathcal{X}$ or $[a,B]\cap\mathcal{X}=\emptyset$.
$\mathcal{X}\sse\mathcal{R}$ is {\em Ramsey null} if for every $\emptyset\ne [a,A]$, there is a $B\in[a,A]$ such that $[a,B]\cap\mathcal{X}=\emptyset$.
\end{defn}

\begin{defn}[\cite{Todo}]\label{defn.5.3}
A triple $(\mathcal{R},\le,r)$ is a {\em topological Ramsey space} if every subset of $\mathcal{R}$ with the property of Baire is Ramsey and every meager subset of $\mathcal{R}$ is Ramsey null.
\end{defn}

The following is the generalization of Ellentuck's Theorem to the general framework of topological Ramsey spaces.

\begin{thm}[Abstract Ellentuck Theorem -- see Theorem 5.4 in \cite{Todo}]\label{thm.AET}\rm 

\it
If $(\mathcal{R},\le,r)$ is closed (as a subspace of $(\mathcal{AR})^{\bN}$) and satisfies axioms {\bf A.1}, {\bf A.2}, {\bf A.3}, and {\bf A.4},
then the triple $(\mathcal{R},\le,r)$ forms a topological Ramsey space.
\end{thm}

\section{Main results}\label{setting}

In this section we will characterize those topological Ramsey spaces $\mathcal R$
which, for each $k<\om$, induce  a 
 topological space $\mathcal R_k\subseteq(\mathcal{AR}_k)^{\mathbb N}$ which, 
with the subspace topology inherited from the 
Tychonoff topology on $(\mathcal{AR}_k)^{\mathbb N}$,
has the property that 
 every Baire subset of $\mathcal R_k$ is Ramsey. 
In our context, each $\mathcal{AR}_k$ is understood as a discrete space.
Our characterization involves augmenting the structure $(\mathcal{R}, \le,r)$ of a typical topological Ramsey space with symbols for the projected spaces $(\mathcal{R}_k)_{k<\om}$, an extra operation symbol $\circ$, and a new finitization function symbol $s$.

 Consider a structure $(\mathcal R, \leq, r, (\mathcal R_k)_{k<\omega}, \circ, s)$.
Let $\mathcal R$,  $\leq$ and $r$ be as  in the previous section. For every $k<\omega$, $\mathcal R_k$ is a nonempty set. 
Every $\mathcal R_k$ will be understood as a projection of $\mathcal R$ to $(\mathcal{AR}_k)^{\mathbb N}$, in a sense that will be made clear  (see Axiom {\bf A.6} and Remark \ref{projection spaces} below).  The symbol $\circ$ denotes an operation $\circ : \mathcal R\times (\mathcal R\cup\bigcup_{k<\omega}\mathcal R_k) \rightarrow (\mathcal R\cup\bigcup_{k<\omega}\mathcal R_k)$. The symbol $s$ denotes a function $s : \mathbb N\times
\bigcup_{k<\omega}\mathcal{R}_k \rightarrow \bigcup_{k<\omega}\mathcal{AR}_k$. We now introduce axioms {\bf A.5}--{\bf A.7} which we will prove in Theorem \ref{main thm} suffice to obtain the characterization announced at the beginning of this paragraph.

\medskip

\subparagraph*{\bf A.5} (Rules for the operation $\circ$).
\begin{itemize}
\item[{ (a)}] For all $A,B\in\mathcal R$, $A\circ B\in\mathcal R$.
\item[{ (b)}] For all $(A,X)\in \mathcal R\times\mathcal R_k$, $A\circ X\in \mathcal R_k$.
\item[{ (c)}] For all $A,B,C\in\mathcal R$, $A\circ (B\circ C) = (A\circ B)\circ C$.
\item[{ (d)}] For all $A,B\in\mathcal R$ and $X\in\mathcal R_k$, $A\circ (B\circ X) = (A\circ B)\circ X$.
\item[{ (e)}] For every $A,B\in\mathcal R$, if there exists $C\in\mathcal R$ such that $B= A\circ C$ then $B\leq A$.
\end{itemize}

\begin{nota}
For every $A\in\mathcal R$ and every $k<\omega$, let
\begin{equation}
\mathcal R_k|A = \{A\circ X : X\in\mathcal R_k\}.
\end{equation}
\end{nota}

\begin{defn}\label{ramsey sets}
A set $\mathcal X\subseteq\mathcal R_k$ is \textbf{Ramsey} if for every $B\in\mathcal R$ there exists $A\in\mathcal R$ with $A\leq B$ such that $\mathcal{R}_k|A\subseteq\mathcal X$ or $\mathcal{R}_k|A\cap\mathcal X=\emptyset$.
\end{defn}

\begin{rem}
Technically, if $k=0$, then $\mathcal{R}_k$ is a singleton, so every subset is Ramsey.
\end{rem}

\subparagraph*{\bf A.6}    (Rules for the function $s$)  For each $k\in\om$, the following hold:

\begin{itemize}
\item[{ (a)}] 
Let  $X\in\mathcal R_k$ be given. 
If $n \ge k$ then $s(n,X)\in\mathcal{AR}_k$.
If $n< k$ then $s(n,X) = r_n(s(k,X))$. For $k>0$, if $a=s(n,X)$ for some $n\geq k$, then we will assume that $r_{k-1}(a)=s(k-1,X)$.
\item[{ (b)}] 
 For all $X\in\mathcal R_k$, $A\in\mathcal R$, the following hold: For all $n\geq k$, we have $s(n,A\circ X)\in\mathcal{AR}_k|A$ and $\depth_A s(n,A\circ X) < \depth_A s(n+1,A\circ X)$.
\item[{ (c)}]
For $X,Y\in \mathcal R_k$,
$s(n,X) = s(m,Y)$ implies $n=m$ and $\forall j<n$, $s(j,X) = s(j,Y)$.
\item[{ (d)}] 
For $X,Y\in \mathcal R_k$, $X\neq Y$ if and only if $\exists n$, $s(n,X)\neq s(n,Y)$.
\end{itemize}

\medskip

By parts (c) and (d) of \bf A.6\rm, 
 each $X\in\mathcal{R}_k$ may be uniquely identified with its sequence $(s(n,X))_{n<\om}$ of $s$-approximations.
By part (a) of \bf A.6\rm, 
the sequence $(s(n,X))_{n\ge k}$  is an element of the infinite product $(\mathcal{AR}_k)^{\mathbb{N}}$.
Moreover, since $s(k,X)$ determines $s(n,X)$ for all $n<k$,
this sequence is uniquely identified with $X$.
Thus,
the set $\mathcal{R}_k$ can be identified with a subset of $(\mathcal{AR}_k)^{\mathbb{N}}$, 
inheriting the subspace topology  from the Tychonoff topology on $(\mathcal{AR}_k)^{\mathbb{N}}$.

\begin{nota}
For $a\in\mathcal{AR}_k$,
let $\langle a \rangle$ denote the set  $\{X\in\mathcal R_k : (\exists n)\, s(n,X)=a\}$.
\end{nota}

The next three facts follow immediately from {\bf A.5} and {\bf A.6}:

\begin{fact1}\label{fact.1} The family of $\langle a \rangle$, $a\in\mathcal{AR}_k$, is a base for the Tychonoff topology on $\mathcal R_k$. 
\end{fact1}

\begin{fact2}\label{fact.2}
 For every $A\in\mathcal R$, $\mathcal R_k|A\subset \bigcup\{\langle a \rangle : a\in \mathcal{AR}_k|A,\ \depth_A(a)>k\}$.
\end{fact2}

\begin{fact3}\label{fact.3}
 For every $A,B\in\mathcal R$, $\{(A\circ B)\circ X : X\in\mathcal R_k \}=\{A\circ (B\circ X) : X\in\mathcal R_k \}\subseteq \mathcal R_k|A.$
\end{fact3}

\begin{nota}
For $m\leq n$ and $A\in\mathcal R$, let $\mathcal{AR}\binom{n}{m}|A=\{a\in\mathcal{AR}_m|A : \depth_A(a)=n \}$,
and let $\mathcal{AR}{n \choose \ge m}|A$ denote $\bigcup_{j\ge m}\mathcal{AR}{n\choose j}|A$.
Also, for any $k<\omega$ and $a,b\in\mathcal{AR}_k$, write $a<b$ if there exists $X\in\mathcal R_k$ and $m<n\in\omega$ such that $a=s(m,X)$ and $b=s(n,X)$. Write, $a\leq b$ if $a<b$ or $a=b$.
\end{nota}

\subparagraph*{\bf A.7}  (Finitization of the operation $\circ$).

Given $A\in\mathcal R$ and $k\leq m\leq n$, the operation $\circ$ can be finitized to a function from $\mathcal{AR}\binom{n}{m}|A\times\mathcal{AR}\binom{m}{k}|A$ onto $\mathcal{AR}\binom{n}{k}|A$, 
 satisfying the following:

\begin{itemize}
\item[{(a)}] Given $a\in \mathcal{AR}\binom{n}{k}|A$ and $b\in \mathcal{AR}_n|A$, if \ $b \circ a < c$ for some $c\in\mathcal{AR}_k|A$ then there exists $b'\in \mathcal{AR}_n|A$ such that $b < b'$ and $c=b'\circ a$.
\item[{(b)}] Given $a\in \mathcal{AR}\binom{n}{k}|A$ and $b,c\in \mathcal{AR}_n|A$, if \ $b < c$ then $b\circ a < c\circ a$.
\item[{(c)}] Let $A\in\mathcal R$, $a\in \mathcal{AR}\binom{n}{k}|A$, and $X\in \mathcal{R}_k$ with $X\in \langle a \rangle$ be given. 
If $n>k$, then $s(n,A\circ X)=r_n(A)\circ a$. 
If $n=k$, then $s(k,A\circ X)=a$.
\end{itemize}

\begin{rem} \label{projection spaces}
{\bf A.6} allows us to identify each $X\in\mathcal R_k$, $k<\omega$, with the sequence $(s(n,X))_{n\geq k}$, and in this way each $\mathcal R_k$ can be regarded as a subspace of $(\mathcal{AR}_k)^{\mathbb N}$ with the Tychonoff topology obtained by endowing $\mathcal{AR}_k$ with the discrete topology. Part (b) of {\bf A.6} indicates that for fixed $k$ and $X\in\mathcal R_k$, the operation $\circ$ and the function $s$ induce a projection map $\pi(A)=A\circ X$, from $\mathcal R$ to $\mathcal{R}_k$. On the other hand, it is worth mentioning at this point that the space $(\mathcal{AR}_k)^{\mathbb N}$ is a Polish metric space and therefore satisfies the Baire Category Theorem stating that the intersection of countably many open dense sets is dense. We shall say that 
$\mathcal{R}_k$ is {\em  metrically closed in} $(\mathcal{AR}_k)^{\mathbb N}$ if for each sequence $(a_n)_{n\ge k}$ in
$(\mathcal{AR}_k)^{\mathbb N}$
satisfying  that $a_m< a_n$, whenever $n>m\ge k$, then $\bigcap_{k\leq n<\om}\langle a_n \rangle= \{X\}$, for some $X\in\mathcal{R}_k$. 
The limit of the sequence $(a_n)_{n\ge k}$, denoted $\lim_{n\ge k}a_n$. 
If $\mathcal{R}_k$ is a closed in  $(\mathcal{AR}_k)^{\mathbb N}$,
then the subspace topology on $\mathcal{R}_k$ inherited from $(\mathcal{AR}_k)^{\mathbb N}$  is
 completely metrizable; and hence, $\mathcal{R}_k$ satisfies the Baire Category Theorem. Notice that if $\mathcal{R}_k$ is closed then for every $A\in\mathcal{R}$, $\mathcal{R}_k|A$ is also closed and satisfies the Baire Category Theorem..
\end{rem}

The following will be used in the sequel.

\begin{lem}\label{end-extension}
$a\leq b$ if and only if $\lgl a\rgl\contains\lgl b\rgl$.
\end{lem}

\begin{proof}
Suppose $a\leq b$. If $a=b$ then we are done, so assume $a<b$. Fix $X\in\mathcal R_k$ and $m<n\in\omega$ such that $a=s(m,X)$ and $b=s(n,X)$. Choose $Y\in\lgl b\rgl$. Then there exists $p\in\omega$ such that $b=s(p,Y)$. Therefore, $s(n,X)=s(p,Y)$. By, part (c) of {\bf A.6}, $n=p$ and $\forall j<n$, $s(j,X)=s(j,Y)$. In particular, $s(m,Y)=s(m,X)=a$. Hence $Y\in \lgl a\rgl$. 
Therefore, $\lgl a\rgl\contains\lgl b\rgl$.

Conversely, suppose that $\lgl a\rgl\contains\lgl b\rgl$,
and  choose $Y\in\lgl b\rgl$. By {\bf A.6}, $Y$ can be identified with the sequence  $(s(m,Y))_{m<\omega}$. Since $Y\in \lgl b\rgl$ and $\lgl a\rgl\contains\lgl b\rgl$, there exist $m,n<\omega$ such that $a=s(m,Y)$ and $b=s(n,Y)$. Notice that $n\geq m$ because otherwise, we would have $\lgl a\rgl\not\contains\lgl b\rgl$.
To see this, 
 supposing toward a contradiction that  $n<m$, it suffices to define $Z\in\mathcal R_k$ such that  $s(j,Z)=s(j,Y)$, for $j\leq n$, and $s(j,Z)=s(m+j,Y)$, for $j>n$. Then $Z\in\lgl b\rgl$ but $Z\not\in\lgl a\rgl$, a contradiction. 
Therefore, it is the case that $n\geq m$,  and we conclude $a=b$ or $a<b$. 
\end{proof}

Now we are ready to state the main result of this article.

\begin{thm}\label{main thm}
Suppose $(\mathcal R, \leq, r, (\mathcal R_k)_{k<\omega}, \circ,s)$ satisfies {\bf A.1} -- {\bf A.7}, $\mathcal R$ is metrically closed in $\mathcal{AR}^{\mathbb N}$ and $\mathcal R_k$ is metrically closed in $\mathcal{AR}_k^{\mathbb N}$, $k<\om$. 
For every  $B\in\mathcal R$, every $k<\om$ and every finite Baire-measurable coloring of $\mathcal R_k$, there exists $A\in\mathcal R$ with $A\leq B$ such that $\{ A\circ X : X\in\mathcal R_k\}$ is monochromatic.
\end{thm}

Thus, Theorem \ref{main thm} implies the following.

\begin{coro}\label{main cor}
Suppose $(\mathcal R, \leq, r, (\mathcal R_k)_{k<\omega}, \circ,s)$ satisfies {\bf A.1} -- {\bf A.7}, $\mathcal R$ is metrically closed in $\mathcal{AR}^{\mathbb N}$ and for all $k<\om$, $\mathcal R_k$ is metrically closed in $\mathcal{AR}_k^{\mathbb N}$. Then
for all $k<\om$, every Baire subset of $\mathcal R_k$ is Ramsey.
\end{coro}

In order to prove Theorem \ref{main thm} we will use the following lemmas. For the proofs of these lemmas, we will assume that $(\mathcal R, \leq, r, (\mathcal R_k)_{k<\omega}, \circ,s)$ satisfies {\bf A.1} - {\bf A.7} and $\mathcal R$ is metrically closed in $\mathcal{AR}^{\mathbb N}$. Notice that Theorem \ref{main thm} is a generalization of the main result in \cite{PromVoi}. The following proofs are based on the techniques used in \cite{PromVoi}.

\begin{lem}\label{lemma open}
Let $A\in\mathcal R$ and $k,m,n\in\omega$ be given, with $m\geq k$.  Let $\mathcal B_i,\ i< n$, be open subsets of $\mathcal R_k$ such that $\bigcup_{i<n}  \mathcal B_i$ is dense. 
Then for each $b\in\mathcal{AR}_m|A$,
 there is a $c\in r_{m+1}[b,A]$
 satisfying that for every $a\in\mathcal{AR}\binom{m+1}{k}|A$ there is an  $i<n$  such that $\lgl c\circ a \rgl\ \subseteq \mathcal B_i$.
\end{lem}

\begin{proof}
Fix $b\in\mathcal{AR}_m|A$ and let $a_0,a_1,\dots,a_{l}$ be an enumeration of $\mathcal{AR}\binom{m+1}{k}|A$. 
Let $b'\in r_{m+1}[b,A]$.
 Since $\langle b'\circ a_0 \rangle$ is open in $\mathcal R_k$ and $\bigcup_{i<n}  \mathcal B_i$ is dense open, 
their intersection is nonempty.
Thus,
by Fact 1, there exists $d_0\in\mathcal{AR}_k|A$ and $j_0<n$ such that  $d_0 > b'\circ a_0$  and $\langle d_0\rangle \subseteq \mathcal B_{j_0}$. 
By part (a) of {\bf A.7}, there exists $c_0\in\mathcal{AR}_{m+1}|A$ such that  $c_0 > b'$ and $d_0=c_0\circ a_0$. 
Thus, $\langle c_0\circ a_0\rangle \subseteq \mathcal B_{j_0}$.
Using Fact 1 and   part (a) of {\bf A.7}, we can inductively build a sequence $b'<c_0< c_1< \dots < c_{l}$ 
and find integers $j_0,j_1,\dots, j_{l}<n$ such that for every $p\le l$, $c_p\in\mathcal{AR}_{m+1}|A$, $\langle c_p\circ a_p\rangle \subseteq \mathcal B_{j_p}$, and $c_p>b'$.
 Let $c=c_{l}$. 
Notice that by  part (b) of {\bf A.7}, for every $p\le l$, $b_p<C$. Then by Lemma \ref{end-extension}, $\langle c\circ a_p\rangle \subseteq \langle b_p\circ a_p\rangle \subseteq \mathcal B_{j_p}$. 

\begin{claim}$r_m(c)=b$.
\end{claim}
\begin{proof}[Proof of the Claim]
 There exists $X\in\mathcal R_{m+1}$ such that $b'=s(m+1,X)$. Therefore, by {\bf A.6}, $b'$ can be identified with the sequence $\{s(j,X)\}_{j\leq n}$. Notice that $r_{m}(b') = s(m,X)$. Since  $b'<c$, there exists $m+1\leq p<q<\omega$ and $Y\in\mathcal R_{m+1}$ such that $b'=s(p,Y)<s(q,Y)=c$. Again, $c$ can be identified with $\{s(j,Y)\}_{j\leq q}$ and $r_{m}(c) = s(m,Y)$. It turns out that $s(m+1, X)=b'=s(p,Y)$. By part (c) of {\bf A.6}, $m+1=p$ and $s(j, X)=s(j,Y)$ for all $j<m+1$. In particular, $b=r_m(b') = s(m,X)=s(m,Y)=r_m(c)$.
\end{proof}

By the Claim, it follows that $b\sqsubset c$ and $c$ is as required.
\end{proof}

\begin{lem}\label{lemma open meager}
Let $B\in\mathcal R$ and $n\in\omega$ be given.
 Let $\mathcal{M}$ be a meager subset of $\mathcal{R}_k|B$, 
and let $\mathcal B_i,\ i< n$, be open subsets of $\mathcal R_k|B$ such that $\bigcup_{i<n} \mathcal B_i$ is dense in $\mathcal R_k|B$.
 Then there is an $A\in\mathcal R$ with $A\leq B$ such that
\begin{enumerate}
\item 
For each $a\in\mathcal{AR}_k|A$ with $\depth_A(a)>k$, there exists an  $i<n$ such that  $\{A\circ X: X\in  \langle a \rangle\cap (\mathcal R_k|B)\}\subseteq \mathcal B_i$; and
\item $\{A\circ X: X\in \mathcal R_k|B \}\cap \mathcal{M} = \emptyset$.
\end{enumerate}
\end{lem}

\begin{proof}
We will use Lemma \ref{lemma open}, relativized to $\mathcal R_k|B$. The proof of the relativized version is analogous, passing to the relative topology and using the fact that if $\mathcal R_k$ is metrically closed in $\mathcal{AR}_k^{\mathbb N}$ then $\mathcal R_k|B$ is also metrically closed in $\mathcal{AR}_k^{\mathbb N}$.  Since $\mathcal{M}\subseteq\mathcal R_k|B$ is meager, there exists a sequence $\mathcal D_m\subseteq\mathcal R_k|B$, $m<\omega$, of dense open sets such that $\mathcal{M}\subseteq\mathcal (R_k|B)\setminus\bigcap_{m<\omega}\mathcal D_m$. For every $m<\omega$, let $\mathcal D_m^*=\bigcap_{l\leq m} \mathcal D_m$. 
Let $b=r_k(B)$.

Since  $\mathcal D_0^*$ and $\bigcup_{i<n}  \mathcal B_i$ are dense open in $\mathcal R_k|B$, $\langle b\rangle$ is open, and $b\in\mathcal{AR}_k|B$,
it follows that there is some $i<n$ for which 
 $\langle b\rangle\cap \mathcal D_0^*\cap \mathcal B_{i}\neq\emptyset$.
 Since $\mathcal D_0^*\cap \mathcal B_{i}$ is open in $\mathcal R_k|B$, by Fact 1, there is a  $b_0\in\mathcal{AR}_k|B$ such that $b_0>b$ and $\langle b_0\rangle \cap (\mathcal R_k|B) \subseteq\ \mathcal D_0^*\cap \mathcal B_{i}$. 

 Let us build a sequence $(b_m)_{m<\omega}$, which in the limit will give us $A$, as follows.
Suppose $b_m$ has been defined. 
By Lemma \ref{lemma open}, there is a  $b_{m+1}\in\mathcal{AR}_{k+m+1}|B$, with $b_m\sqsubset b_{m+1}$, such that for every $a\in\mathcal{AR}\binom{k+m+1}{k}|B$,
 there exists some $i<n$ such that $\langle b_{m+1}\circ a\rangle\cap (\mathcal{R}_k|B) \subseteq \mathcal B_{i}\cap \mathcal D_m^*$.
 Let $A=\lim_m b_m$.
Then $A\in \mathcal R$.

We claim that  $A$ is as required. 
To see this,
let $a\in\mathcal{AR}_k|A$ such that  $\depth_A(a)>k$ be given,
and let $m$ be such that  $k+m+1=\depth_B(a)$. 
By our construction,
there is some $i<n$ such that
 $\langle b_{m+1}\circ a\rangle \cap (\mathcal R_k|B) \subseteq \mathcal B_{i}$. 
Since $b_{m+1}=r_{k+m+1}(A)$ and $a\in\mathcal{AR}{k+m+1\choose k}|A$,
for each
$X\in \langle a \rangle$,
 it follows from
 part (c) of {\bf A.7}  that
$s(k+m+1, A\circ X)=b_{m+1}\circ a$. 
Thus, 
\begin{equation}\label{eq.1}
A\circ X\in\langle b_{m+1}\circ a\rangle\cap (\mathcal{R}_k|B) 
\subseteq
 \mathcal B_{i}\cap \mathcal D_m^*.
\end{equation}
In particular, (1) holds.

We now check that (2) holds.
Let $X\in\mathcal{R}_k|B$ be given.
Let $m<\om$ be given, and let $a=s(k+m+1,X)$.
Then $X\in\lgl a\rgl$, and $k+m+1=\depth_B(a)$.
By Equation (\ref{eq.1}),
$A\circ X\in\mathcal{D}^*_m$.
Since this holds for all $m<\om$,
we find that $A\circ X\in\mathcal{M}$.
\end{proof}

Since $\mathcal R$ is a topological Ramsey space the following analog of Ramsey's Theorem is true (see \cite{Mijares,Todo}).

\begin{lem}\label{Ramsey thm}
For every $B\in\mathcal R$ and every finite coloring of $\mathcal{AR}_k$, there exists $A\in\mathcal R$ with $A\leq B$ such that $\mathcal{AR}_k|A$ is monochromatic.
\end{lem}

\begin{rem}
{\bf A.1} - {\bf A.4} and the assumption that $\mathcal R$ is metrically closed in $\mathcal{AR}^{\mathbb N}$ are sufficient for the proof of Lemma \ref{Ramsey thm}. In fact, Lemma \ref{Ramsey thm} is a special case of the Abstract Nash-Williams Theorem (see \cite{Todo}), which follows from the Abstract Ellentuck Theorem.
\end{rem}

Now, let us prove our main result.

\begin{proof}[Proof of Theorem \ref{main thm}]
Fix $B\in\mathcal R$. Given $n<\omega$, let $c : \mathcal R_k|B \rightarrow n$ be a Baire-measurable coloring. 
Then there exist open sets $\mathcal B_i\subseteq\mathcal R_k|B$, $i<n$, such that the sets 
$$\mathcal{M}_i := (c^{-1}(\{i\})\setminus \mathcal B_i)\cup (\mathcal B_i\setminus c^{-1}(\{i\}))$$ 
are meager in $\mathcal R_k|B$.
Let $\mathcal{M}=\bigcup_{i<n} \mathcal{M}_i$. Then $(\mathcal R_k|B)\setminus \mathcal{M}\subseteq\bigcup_{i<n} \mathcal B_i$. 
Thus, since $\mathcal R_k|B$ satisfies the Baire Category Theorem, $\bigcup_{i<n} \mathcal B_i$ is dense in $\mathcal R_k|B$. 
Choose $A\in\mathcal R$ as in Lemma \ref{lemma open meager} applied to $\mathcal{M}$ and the $\mathcal B_i$'s;
 and let $A_0\le A$ such that $\depth_A(r_k(A_0))>k$.
It follows that for every $a\in\mathcal{AR}_k|A_0$, 
there exists $i<n$ such that $\{A_0\circ X : X\in\ \langle a \rangle\cap\ \mathcal R_k|B\}\subseteq \mathcal B_i\cap c^{-1}(\{i\})$. 
In particular, for every $a\in\mathcal{AR}_k|A_0$, $c$ is constant on $\{A_0\circ X : X\in\ \langle a \rangle\cap\ \mathcal R_k|B\}$.

Define $\hat{c}: \mathcal{AR}_k|A_0 \rightarrow n$ by $\hat{c}(a) = c(A_0\circ X)$, for any $X\in\ \langle a \rangle\cap\ \mathcal R_k|B$. 
By Lemma \ref{Ramsey thm}, there exists $A_1\leq A_0$ such that $\hat{c}$ is constant on $\mathcal{AR}_k|A_1$. 
By the definition of $\hat{c}$, it follows that 
for all  $X',Y'\in\mathcal{R}_k|A_1\circ B$,
there are some $a,a'\in\mathcal{AR}_k|A_1$ such that $X'\in\lgl a\rgl$ and $Y'\in\lgl a'\rgl$;
and thus $c(A_0\circ X')=
c(A_0\circ Y')$.
 By Fact  3,  the set $\{(A_0\circ A_1)\circ X : X\in\mathcal R_k|B\} = \{(A_0\circ A_1)\circ (B\circ  X) : X\in\mathcal R_k\}$. 
Letting  $A=A_1\circ A_0\circ B$, by Fact 2, we see that  $c$ is 
is monochromatic on $\mathcal{R}_k|A$.

\end{proof}

\section{Generalized ascending parameter words and block sequences}\label{general parameter}

 We introduce a class of topological Ramsey spaces which generalizes the spaces of ascen-\linebreak ding parameter words studied by Carlson-Simpson \cite{CarSim2} and  Pr\"omel-Voigt \cite{PromVoi} (see Section \ref{ascending}), and which turns out to be also a generalization of the spaces $\FIN_m^{[\infty]}$, $0<m<\omega$, of block sequences defined by Todorcevic \cite{Todo}. We show that each element of this class admits projection spaces where every Baire set is Ramsey, fitting into the abstract setting introduced in Section \ref{setting}. In order to show that our space is a topological Ramsey space, we use an infinitary version of the Hales-Jewett Theorem to deduce a pigeon hole principle which generalizes Gowers' Theorem \cite{gowers}.

\subsection{Generalized ascending parameter words}

Let $X$ and $Y$ be two nonempty sets of integers. Given a set $S\subseteq X\times Y$, let $\dom(S)=\{i\in X : (\exists j\in Y)\ (i,j)\in S\}$ and $\ran(S)=\{j\in Y : (\exists i\in X)\ (i,j)\in S\}$. As customary, we will identify each integer $m>0$ with the set $\{0,1,\dots,m-1\}$. Let $\omega$ be the set of nonnegative integers. Given $t,m<\omega$, with $m>0$, and $\alpha\leq\beta\leq\omega$, let $\mathcal S_t^{<}\binom{\beta,m}{\alpha}$ denote the set of all the surjective functions $A: (t+\beta)\times m \rightarrow t+\alpha$ satisfying

\begin{enumerate}
\item $A(i,l)=i$ for every $i<t$ and every $l<m$.
\item For all $j<\alpha$, $A^{-1}(\{t+j\})$ is a function; that is, for all $i\in\dom(A^{-1}(\{t+j\}))$ there exists a unique $l< m$ such that $(i,l)\in A^{-1}(\{t+j\})$.
\item For all $j<\alpha$, $\dom A^{-1}(\{t+j\})$ is a finite set.
\item For all $j<t+\alpha$, $m-1\in \ran(A^{-1}(\{j\}))$.
\item $\min\dom A^{-1}(\{i\})< \min \dom A^{-1}(\{j\})$ for all $i<j<t+\alpha$.
\item $\max \dom A^{-1}(\{t+i\})<\min \dom A^{-1}(\{t+j\})$ for all $i<j<\alpha$.
\end{enumerate}

\subparagraph*{\bf The tetris operation} For $S\subseteq (t+\beta)\times m$, let $T(S)=\{(i,\max\{0,j-1\}) : (i,j)\in S\}$. For $l<\omega$, let us define $T^{l}(S)$ recursively, as follows. $T^0(S)=S$, $T^1(S)= T(S)$ and $T^{l+1}(S)=T(T^l(S))$.

\medskip

\subparagraph*{\bf The composition} For $A\in\mathcal S_t\binom{\gamma,m}{\beta}$  and $B\in\mathcal S_t\binom{\beta,m}{\alpha}$, the operation $A\cdot B\in\mathcal S_t\binom{\gamma,m}{\alpha}$ is defined by $(A\cdot B)(i,j) = B(A(i,j),m-1)$. 

\begin{rem}
In Theorem \ref{general parameter space} below we will prove that $\mathcal S_t^{<}\binom{\omega,m}{\omega}$ is a topological Ramsey space. Notice that for $t=0$, $\mathcal S_0^{<}\binom{\omega,1}{\omega}$ is essentially the set of infinite subsets of $\omega$, so as a topological Ramsey space $\mathcal S_0^{<}\binom{\omega,1}{\omega}$ will coincide with Ellentuck's space \cite{Ellen}; and $\mathcal S_0^{<}\binom{\omega,m}{\omega}=\emptyset$, for $m>1$. So we will assume $t>0$ throughout the rest of this section.
\end{rem}
\begin{rem}\label{block seq} Let $0<m<\omega$ be given. For a function $p\colon\omega\to \{0,1,\dots, m\}$, let $\supp(p)=\{i\in\omega : p(i)\neq 0\}$. Denote by $\FIN_m$ the collection of all the functions $p\colon\omega\to \{0,1,\dots, m\}$ such that $\supp(p)$ is finite and $m\in \ran(p)$. A \textbf{block sequence} of elements of $\FIN_m$ is a sequence $(p_n)_{n<\omega}$ such that $\max\ \supp(p_n)<\min\ \supp(p_{n+1})$,  for all $n<\omega$. Let $\FIN_m^{[\infty]}$ be the collection of all such block sequences. Notice that for all $0<m<\omega$, $\mathcal S_1^{<}\binom{\omega,m}{\omega}$ can be identified with $\FIN_m^{[\infty]}$: A block sequence $P= (p_n)_{n<\omega}\in \FIN_m^{[\infty]}$ determines a function $A_P\in \mathcal S_1^{<}\binom{\omega,m}{\omega}$ defined as follows:

\[A_P(i,j) =\left\{ \begin{array}{rcl}
0\ \ \ \  & \mbox{if} & (\forall n<\omega)\ i\notin\supp(p_n),\\
0\ \ \ \  & \mbox{if} & (\exists n<\omega)\ i\in\supp(p_n)\ \wedge\ p_n(i)\neq j+1,\\
n+1\ & \mbox{if} & i\in\supp(p_n)\ \wedge\ p_n(i)=j+1.\\
\end{array}
\right. \]

Conversely, a function $A\in \mathcal S_1^{<}\binom{\omega,m}{\omega}$ determines a block sequence $P_A= (p_n)_{n<\omega}\in \FIN_m^{[\infty]}$ where, for each $n<\omega$, $p_n$ is  given by \[p_n(i) =\left\{ \begin{array}{rcl}
0\ \ \ \  \ \ \ & \mbox{if} &  i\notin\dom(A^{-1}(\{n+1\})),\\
j+1\ \ \ \  & \mbox{if} & (i,j)\in A^{-1}(\{n+1\}).\\
\end{array}
\right. \]

\end{rem}

\subsection{A topological Ramsey space of generalized ascending parameter words}

The purpose of this section is to prove that $\mathcal S_t^{<}\binom{\omega,m}{\omega}$ is a topological Ramsey space. Define the function $r$ on $\mathbb N\times\mathcal S_t^{<}\binom{\omega,m}{\omega}$ as  $r(n,A) = \emptyset$, if $n= 0$ and 

\[r(n,A) = A\upharpoonright \{(i,l)\in \bigcup_{j<t+n} A^{-1}(\{j\}): i<\min \dom A^{-1}(\{t+n\})\}, \ \ \mbox{if}\ n> 0.\]

\medskip

Let $A\in \mathcal S_t^{<}\binom{\omega,m}{\omega}$ be given. For every $n<\omega$, let $a_n=A^{-1}(\{n\})$, and write $A = \{a_0,a_1,\dots \}$. Let $T$ denote the tetris operation. For $t\leq n<\omega$ and $l< t+m$, define

\[S^l(a_n) =\left\{ \begin{array}{rcl}
T^{l-t}(a_n)\ \ & \mbox{if} &  t\leq l< t+m,\\
& & \\
a_n\cup a_l \ \ & \mbox{if} &  l<t.\\
\end{array}
\right. \]

\noindent Here $a_n\cup a_l$ is the union of $a_n$ and $a_l$ as subsets of $(t+\omega)\times m$.

\medskip

Let $[A]$ denote the collection of all the symbols of the form \linebreak$S^{l_1}(a_{n_1}) + S^{l_2}(a_{n_2}) +\cdots + S^{l_q}(a_{n_q})$ such that $n_i\geq t$ and $l_i< t+m$, for all $i\in\{1,\dots,q\}$, and at least one of the $l_i$'s is equal to $t$. 

We shall identify each $S^{l_1}(a_{n_1}) + S^{l_2}(a_{n_2}) +\cdots + S^{l_q}(a_{n_q})\in [A]$ with a function $f\in\mathcal S_t^{<}\binom{e,m}{1}|A$, for some $0<e<\omega$, as follows. 

Suppose that there exists $i\in\{1,\dots, q\}$ such that $l_i<t$. Let $i_1<\dots<i_p$ be an increasing enumeration of all such $i$'s. Let $j_0\in\{1,\dots, q\}$ be such that $l_{j_0}=t$. Then $S^{l_{j_0}}(a_{n_{j_0}})=a_{n_{j_0}}$. Let $e=\min\dom A^{-1}(\{t+n_{j_0}+1\})$. Define the surjective function $f: (t+e)\times m\rightarrow t+1$ by setting

\begin{enumerate}
\item $f^{-1}(\{j\}) = (t+e)\times m\cap a_j$ for all $j<t$ with $j\notin\{l_{i_1},\dots, l_{i_p}\}$.
\item $f^{-1}(\{l_{i_1}\})=(t+e)\times m\cap\bigcup\{a_n :n\geq t, n\notin\{n_1,\dots, n_q\}\}\cup S^{l_1}(a_{n_1})\cup \cdots \cup S^{l_{i_1}}(a_{n_{i_1}})\setminus a_{n_{j_0}} $.
\item $f^{-1}(\{l_{i_{d+1}}\})= (t+e)\times m\cap S^{l_{i_d+1}}(a_{n_{i_d+1}})\cup \cdots \cup S^{l_{i_{d+1}}}(a_{n_{i_{d+1}}})\setminus a_{n_{j_0}}$, $1\leq d<p$.
\item If there exists $i\in\{1,\dots, q\}$ such that $i>i_p$ then let $f^{-1}(\{t\})= (t+e)\times m\cap S^{l_{i_p+1}}(a_{n_{i_p+1}})\cup \cdots \cup S^{l_{i_q}}(a_{n_q})\cup a_{n_{j_0}}$. Otherwise, let $f^{-1}(\{t\})= (t+e)\times m\cap a_{n_{j_0}}$.
\end{enumerate}

\medskip

If for all $i\in\{1,\dots, q\}$ we have $l_i\geq t$ then set

\begin{enumerate}
\item $f^{-1}(\{0\}) = (t+e)\times m\cap a_0\cup \bigcup\{a_n :n\geq t, n\notin\{n_1,\dots, n_q\}\}$.
\item $f^{-1}(\{j\}) = (t+e)\times m\cap a_j$ for all $0<j<t$.
\item $f^{-1}(\{t\})= (t+e)\times m\cap S^{l_1}(a_{n_1}) + S^{l_2}(a_{n_2}) +\cdots + S^{l_q}(a_{n_q})$.
\end{enumerate}

\medskip

\noindent This finishes the definition of $f$. The condition that at least one of the $l_i$'s is equal to $t$ ensures that $f\in\mathcal S_t^{<}\binom{e,m}{1}$. On the other hand, given $0<e<\omega$, every $f\in\mathcal S_t^{<}\binom{e,m}{1}$ can be represented as an element of $[A]$ for some $A$.

\medskip

\subparagraph*{\bf The quasi-order} If $B\in\mathcal S_t^{<}\binom{\omega,m}{\omega}$ then write $B\leq A$ if $[B]\subseteq [A]$. It can be easily proved that  if there exists $C\in\mathcal S_t^{<}\binom{\omega,m}{\omega}$ such that $B= A\cdot C$ then $B\leq A$. 

\medskip

The following is the main result of this Section. 

\begin{thm}\label{general parameter space}
For every $t>0$ and every $m>0$, $(\mathcal S_t^{<}\binom{\omega,m}{\omega}, \leq, r)$ is a topological Ramsey space.
\end{thm}

The case $t=1$ of Theorem \ref{general parameter space} was proved by Todorcevic and in virtue of Remark \ref{block seq} it can be restated as follows.

\begin{thm}[Todorcevic,  Theorem 5.22 in \cite{Todo}]\label{FIN_m}
For every $m>0$, $(\mathcal S_1^{<}\binom{\omega,m}{\omega}, \leq, r)$ is a topological Ramsey space.
\end{thm}

Let $\mathcal S_t^{<}\binom{<\omega,m}{<\omega}$ denote the range of the function $r$. That is, $\mathcal S_t^{<}\binom{<\omega,m}{<\omega}=\mathcal{AR}$ for the space $\mathcal R=\mathcal S_t^{<}\binom{\omega,m}{\omega}$. Notice that $\mathcal S_t^{<}\binom{<\omega,m}{<\omega}=\bigcup_{l\leq n<\omega}\mathcal S_t^{<}\binom{n,m}{l}$.  Before proving Theorem \ref{general parameter space} for $t>1$, we will establish a result concerning finite colorings of $\mathcal S_t^{<}\binom{<\omega,m}{<\omega}$ in Theorem \ref{pigeon hole} below, from which {\bf A.4} for $\mathcal S_t^{<}\binom{\omega,m}{\omega}$ can be easily deduced. In order to prove Theorem \ref{pigeon hole}, we will need an infinitary version of the Hales-Jewett Theorem.

\medskip

Let $L$ be a finite alphabet and let $v\notin L$. Let $W_{L}$ denote the collection of words over $L$, and  let $W_{L\cup\{v\}}$ denote the collection of words $w(v)$ over the alphabet $L\cup\{v\}$, such that the symbol $v$ appears at least once in $w(v)$. The elements of $W_{L\cup\{v\}}$ are called \textbf{variable words}.

Given $w(v)\in W_{L\cup\{v\}}$, the symbol $|w(v)|$ denotes the length of $w(v)$. Also, if $l\in L\cup\{v\}$,  let $w(l)$ be the word in $W_{L}\cup W_{L\cup\{v\}}$ obtained by replacing every occurrence of $v$ in $w(v)$ by $l$. For an infinite sequence of words $(w_n(v))_{n<\omega}$, let $$[(w_n(v))_{n<\omega}]_{L\cup\{v\}}=\{w_{n_0}(l_0)w_{n_1}(l_1)\dots w_{n_q}(l_q)\in W_{L\cup\{v\}} : n_1<\dots<n_q; l_i\in L\cup\{v\}\ (i\leq q)\}.$$

 \noindent An infinite sequence $(w_n(v))_{n<\omega}$, with $w_n(v)\in W_{L\cup\{v\}}$, is \textbf{rapidly increasing} if for every $n<\omega$, $|w_n(v)|>\sum_{i<n}|w_i(v)|$. 

\begin{thm}[Infinite Hales-Jewett Theorem. See \cite{Todo}, Theorem 4.21]\label{infinite HJ}
Let $L$ be a finite alphabet and let $v$ be a symbol which is not in $L$. Let $(w_n(v))_{n<\omega}$ be an infinite rapidly increasing sequence of variable words. Then for every finite coloring of $W_{L\cup\{v\}}$ there exists an infinite rapidly increasing sequence $(u_j(v))_{j<\omega}$ with $u_j(v)\in [(w_n(v))_{n<\omega}]_{L\cup\{v\}}$, $j<\omega$, such that $[(u_j(v))_{j<\omega}]_{L\cup\{v\}}$ is monochromatic.
\end{thm}

From Theorem \ref{infinite HJ} we can prove the following.

\begin{thm}\label{pigeon hole}
Let $A\in \mathcal S_t^{<}\binom{\omega,m}{\omega}$ be given. For every finite coloring of $[A]$ there exists $B\in\mathcal S_t^{<}\binom{\omega,m}{\omega}$ such that $B\leq A$ and $[B]$ is monochromatic.
\end{thm}

\begin{proof}

Let us write, $A = \{a_0,a_1,\dots \}$, where $a_n=A^{-1}(\{n\})$. Consider the alphabet \linebreak$L=t+m\setminus\{t\}$, and let $v=t$. Fix a finite coloring $c$ of $[A]$. Also, let $(w_n(v))_n$ be the infinite rapidly increasing sequence of variable words such that for $n<\omega$ we have $w_n(v)=vv\dots v$, where $v$ appears exactly $2^n$ times. Now, define the coloring $c'$ on $[(w_n(v))_n]_{L\cup\{v\}}$ by 

$$c'(w_{n_0}(l_0)w_{n_1}(l_1)\dots w_{n_q}(l_q)) = c(S^{l_0}(a_{n_0}) + S^{l_1}(a_{n_1}) +\cdots + S^{l_q}(a_{n_q}))$$ 

\medskip

\noindent and let $c'(u)=0$, if $u\in W_{L\cup\{v\}}\setminus [(w_n(v))_n]_{L\cup\{v\}}$. 

Given $w_{n_0}(l_0)w_{n_1}(l_1)\dots w_{n_q}(l_q)\in [(w_n(v))_n]_{L\cup\{v\}}$, at least one $l_i$ must be equal to $v$ (that is, equal to $t$). Then $S^{l_0}(a_{n_0}) + S^{l_1}(a_{n_1}) +\cdots + S^{l_q}(a_{n_q})$ is in fact an element of $[A]$, and it is uniquely determined by $w_{n_1}(l_1)w_{n_2}(l_2)\dots w_{n_q}(l_q)$ because $(w_n(v))_{n<\omega}$ is rapidly increasing. Apply Theorem \ref{infinite HJ}, and let $(u_j(v))_{j<\omega}$, with  $u_j(v)\in [(w_n(v))_n]_{L\cup\{v\}}$, be an infinite rapidly increasing sequence of variable words such that $[(u_j(v))_j]_{L\cup\{v\}}$ is monochromatic for $c'$. Let us say that the constant color is $p$. Define $B= (b_0,b_1, \dots)$ in this way: For every $j<\omega$, if $$u_j(v)=w_{n_0}(l_0)w_{n_1}(l_1)\cdots w_{n_j}(l_{q_j})$$ then $$b_j=S^{l_0}(a_{n_0}) + S^{l_1}(a_{n_1})  +\cdots + S^{l_{q_j}}(a_{n_{q_j}}).$$ 

Again, here $b_j$ stands for $B^{-1}(\{j\})$. Thus, $[B]$ is monochromatic for $c$: If $b\in [B]$ and $b= S^{l_0}(b_{j_0}) + S^{l_1}(b_{j_1}) +\cdots + S^{l_s}(b_{j_s})$ with $b_{j_i}=S^{l_i^0}(a_{n_i^0}) + S^{l_i^1}(a_{n_i^1}) +\cdots + S^{l_i^{q_i}}(a_{n_i^{q_i}})$, $0\leq i\leq s$, then

$$c(b)=c(S^{l_0}(b_{j_0})  +\cdots + S^{l_s}(b_{j_s})) = c(S^{l_0}(\sum_{0\leq j\leq q_1}S^{l_0^j}(a_{n_0^j})) +\cdots + S^{l_s}(\sum_{0\leq j\leq q_s}S^{l_s^j}(a_{n_s^j})))$$ $$\ \ \ \ \ \ \ \ \ \ \ \ \ = c'((w_{n_0^0}(l_0^0)\dots w_{n_0^{q_1}}(l_0^{q_1}))(l_0)\dots (w_{n_s^0}(l_s^0)\dots w_{n_s^{q_s}}(l_0^{q_s}))(l_s))= p.$$
\end{proof}

\begin{proof}[Proof of Theorem \ref{general parameter space}]

The quasi-order $\leq$ defined on $\mathcal S_t^{<}\binom{\omega,m}{\omega}$ admits a natural finitization $\leq_{\fin}$ defined on $\mathcal S_t^{<}\binom{<\omega,m}{<\omega}$ as follows. Consider $a,b\in\mathcal S_t^{<}\binom{<\omega,m}{<\omega}$. Write $a= \{a_0, \dots, a_{n-1}\}$ and $b=\{b_0,\dots,b_{p-1}\}$, where $0<n,p<\omega$, $a_j=a^{-1}(\{j\})$ and $b_j=b^{-1}(\{j\})$. Let $[b]$ denote the collection of all the sets of the form $S^{l_1}(b_{n_1}) + S^{l_2}(b_{n_2}) +\cdots + S^{l_q}(b_{n_q})$ satisying that if $n_i\geq t$ for all $i\in\{1,\dots,q\}$, then at least one of the $l_i$'s is equal to $t$. Write $a\leq_{\fin} b$ if and only if $\dom(a)=\dom(b)$ and $a_j\in [b]$, $j<n$. It is easy to see that with these definitions, {\bf A.1}-{\bf A.3} are satisfied, and to prove that $\mathcal S_t^{<}\binom{\omega,m}{\omega}$ is metrically closed in $(\mathcal S_t^{<}\binom{<\omega,m}{<\omega})^{\mathbb N}$. Also, {\bf A.4} for $\mathcal S_t^{<}\binom{\omega,m}{\omega}$ follows easily from Theorem \ref{pigeon hole}. This completes the proof of Theorem \ref{general parameter space}.
\end{proof}

\subsection{Baire sets of generalized ascending parameter words are Ramsey.}

Given $k,m,t<\omega$ with $m>0$. In this Section we will prove that the Baire subsets of $\mathcal S_t^{<}\binom{\omega,m}{k}$ are Ramsey. Let $\pi: S_t^{<}\binom{\omega,m}{\omega}\rightarrow S_t^{<}\binom{\omega,m}{k}$ be defined as follows: 

\[\pi(A)(i,j) =\left\{ \begin{array}{rcl}
0\ \ \  & \mbox{if} & A(i,j)\geq t+k\\
& & \\
A(i,j) & \mbox{if} & 0\leq A(i,j)< t+k\\
\end{array}
\right. \]

Notice that $\pi$ is a surjection. For $n>k$, we will extend the function $\pi$ to $\mathcal{AR}_n=\mathcal S_t^{<}\binom{<\omega,m}{n}$ as follows. Given $a\in\mathcal{AR}_n$ and $(i,j)$ in the domain of $a$, let

\[\pi(a)(i,j) =\left\{ \begin{array}{rcl}
0\ \ \  & \mbox{if} & t+k\leq a(i,j)< t+n\\
& & \\
a(i,j) & \mbox{if} & 0\leq a(i,j)< t+k\\
\end{array}
\right. \]

Define the function $s$ on $\mathbb N\times\bigcup_k\mathcal S_t^{<}\binom{\omega,m}{k}$ as follows. Given $n<\omega$, $X\in\mathcal S_t^{<}\binom{\omega,m}{k}$ and any $A\in\mathcal S_t^{<}\binom{\omega,m}{\omega}$ such that $\pi(A)=X$, let 

\[s(n,X) =\left\{ \begin{array}{rcl}
r_n(A)\ \ \ \ \ \ \  & \mbox{if} & 0\leq n\leq k\\
\pi(r_n(A)) \ \ \ \ & \mbox{if} & n>k\\
\end{array}
\right. \]

\bigskip

For $A,B\in\mathcal S_t^{<}\binom{\omega,m}{\omega}$ and $X\in\mathcal S_t^{<}\binom{\omega,m}{k}$, $k<\omega$, define $A\circ B = A\cdot B$ and $A\circ X= A\cdot X$. We will prove that the structure $(\mathcal S_t^{<}\binom{\omega,m}{\omega}, \leq, r, (\mathcal S_t^{<}\binom{\omega,m}{k})_k,\circ,s)$ satisfies axioms ${\bf A.5- A.7}$. Therefore, by Theorem \ref{main thm}, we will obtain the following

\begin{thm}\label{block sequence ramsey} Let $k<\omega$ and $B\in\mathcal S_t^{<}\binom{\omega,m}{\omega}$ be given. For every finite Baire-measurable coloring of $\mathcal S_t^{<}\binom{\omega,m}{k}$, there exists $A\in\mathcal S_t^{<}\binom{\omega,m}{\omega}|B$ such that $\mathcal S_t^{<}\binom{\omega,m}{k}|A$ is monochromatic.
\end{thm}

\begin{claim}
$(\mathcal S_t^{<}\binom{\omega,m}{\omega}, \leq, r, (\mathcal S_t^{<}\binom{\omega,m}{k})_k,\circ,s)$ satisfies axioms ${\bf A.5- A.7}$.
\end{claim}
\begin{proof}
This follows from the definitions. For instance, to show part (c) of ${\bf A.5}$ notice that $$(A\circ B)\circ C\ (i,j)$$ $$\ \ \ \ = C((A\circ B)(i,j),m)$$ $$\ \ \ \ \ \ \ =C((B(A(i,j),m)),m)$$ $$\ = B\circ C(A(i,j),m)$$ $$\ \ =A\circ (B\circ C)\ (i,j) .$$

And in order to prove part (b) of ${\bf A.6}$, for instance, notice that if $A\in\mathcal S_t^{<}\binom{\omega,m}{\omega}$ and $X,Y\in\mathcal S_t^{<}\binom{\omega,m}{k}$ are such that $Y=A\circ X$ then there exists $B\in\mathcal S_t^{<}\binom{\omega,m}{\omega}$ such that $B\leq A$ and $Y=\pi(B)$. Then, from the definition of $s$ and $\pi$ we conclude that $\depth_Bs(n,Y)<\depth_Bs(n+1,Y)$, for all $n<\omega$. Therefore, $\depth_As(n,Y)<\depth_As(n+1,Y)$, for all $n<\omega$.

On the other hand, the definition of the operation $A\circ B =A\cdot B$ was done on $\mathcal S_t^{<}\binom{\beta,m}{\alpha}$, for all $\alpha\leq\beta\leq\omega$. Thus, the finitization asked for in ${\bf A.7}$ was defined at the same time. Parts (a) of ${\bf A.7}$ follow from the definition of the operation $\circ$ by an easy extension of functions. Part (b) is straight forward. Let us prove part (c). Let $A\in \mathcal S_t^{<}\binom{\omega,m}{\omega}$, $a\in \mathcal S_t^{<}\binom{<\omega,m}{k}|A$ and $X\in \mathcal S_t^{<}\binom{\omega,m}{k}|A$ be given, with $\depth_A(a)=n$ and $X\in\ \langle a \rangle$. Define $B\in \mathcal S_t^{<}\binom{\omega,m}{\omega}$ by

\[B(i,j) =\left\{ \begin{array}{rcl}
(A\circ X)(i,j) & \mbox{if} & A(i,j)<t+n\\
A(i,j) \ \ \ \  & \mbox{if} & A(i,j)\geq t+n\\
\end{array}
\right. \]

\noindent Then $B\leq A$ and $\pi(B)=A\circ X$. Notice that 
$$s(n,A\circ X)=r_k(B)=B\upharpoonright \min\ \supp B^{-1}(\{t+k\})= A\circ X\upharpoonright \min\ \supp A^{-1}(\{t+n\})=r_n(A)\circ\ a.$$ 

This completes the proof of the Claim and of Theorem \ref{block sequence ramsey}.
\end{proof}

\begin{rem} The case $t=1$ of Theorem \ref{block sequence ramsey} is just the corresponding version of the infinite Ramsey Theorem \cite{Ramsey} for the topological Ramsey space $\FIN_m^{[\infty]}$ proved by Todorcevic \cite{Todo}. Off course, in the case $t=1$, Theorem \ref{block sequence ramsey} holds for all finite colorings (that is, it is not necessary to restrict to Baire-measurable colorings). But for $t>1$, using the Axiom of Choice, it is possible to define a not Baire-measurable finite coloring of $\mathcal S_t^{<}\binom{\omega,m}{k}$ with no monochromatic set of the form $\mathcal S_t^{<}\binom{\omega,m}{k}|A$.
\end{rem}

\section{Classical examples}\label{classical}
Throughout this Section, we will explorer other examples which fit the abstract setting introduced in this Section \ref{setting}. These classical examples were originally introduced in \cite{CarSim2,PromVoi,Todo,Voigt} and motivated this research. 

\subsection{Parameter words}\label{parameter}

Given $t<\omega$ and ordinals $\alpha\leq\beta\leq\omega$, let $\mathcal S_t\binom{\beta}{\alpha}$ denote the set of all surjective functions $A: t+\beta \rightarrow t+\alpha$ satisfying

\begin{enumerate}
\item $A(i)=i$ for every $i<t$.
\item $min\ A^{-1}(\{i\})< min\ A^{-1}(\{j\})$ for all $i<j<t+\alpha$.
\end{enumerate}

For $A\in\mathcal S_t\binom{\gamma}{\beta}$  and $B\in\mathcal S_t\binom{\beta}{\alpha}$, the composite $A\cdot B\in\mathcal S_t\binom{\gamma}{\alpha}$ is defined by $(A\cdot B)(i) = B(A(i))$.

\medskip

Fix $t<\omega$. Let $\mathcal R = \mathcal S_t\binom{\omega}{\omega}$ and for every positive $k<\omega$, let $\mathcal R_k =\mathcal S_t\binom{\omega}{k}$. Define the operation $\circ : \mathcal R\times (\mathcal R\cup\bigcup_k\mathcal R_k) \rightarrow (\mathcal R\cup\bigcup_k\mathcal R_k)$ as $A\circ B = A\cdot B$. For $A,B\in\mathcal R$, write $A\leq B$ whenever there exists $C\in\mathcal R$ such that $A=B\circ C$. 

At this point, it is useful to understand $ \mathcal S_t\binom{\omega}{\omega}$ as the set of equivalence relations on $t+\omega$ with infinitely many equivalence classes such that for each $A\in \mathcal S_t\binom{\omega}{\omega}$, the restriction $A\upharpoonright t$ is the identity relation on $t=\{0,1,\dots,t-1\}$ (when $t>0$).  Define the function $r:\mathbb N\times\mathcal R \rightarrow \mathcal{AR}$ as: 

\[r(n,A) =\left\{ \begin{array}{rcl}
\emptyset\ \ \ \ \ \ \ \ \ \ \ \ \ \ \ \ \ \ \ \  & \mbox{if} & n= 0\\
& & \\
A\upharpoonright \min\ A^{-1}(\{t+n\}) & \mbox{if} & n> 0\\
\end{array}
\right. \]

The definition of $\mathcal S_t\binom{\beta}{\alpha}$, for $t<\omega$ and ordinals $\alpha\leq\beta\leq\omega$, was taken from \cite{PromVoi}. But the proof of the following is due to Carlson and Simpson \cite{CarSim}:

\begin{thm}[Carlson-Simpson \cite{CarSim}]\label{parameter space} $(\mathcal S_t\binom{\omega}{\omega}, \leq, r)$ is a topological Ramsey space.
\end{thm}

\begin{rem}
Axiom {\bf A.4} for the space $(\mathcal S_t\binom{\omega}{\omega}, \leq, r)$ follows from the infinite version of Ramsey's theorem for parameter words due to Graham and Rothschild  \cite{gralero} (case $k=1$ of  Theorem A in \cite{PromVoi}, page 191).
\end{rem}

\medskip

Similarly, $S_t\binom{\omega}{k}$ can be understood as the set of equivalence relations on $t+\omega$ with exactly $k$ classes disjoint from $t$ such that for each $X\in \mathcal S_t\binom{\omega}{k}$, the restriction $X\upharpoonright t$ is the identity relation on $t=\{0,1,\dots,t-1\}$ (when $t>0$). Let $\pi: S_t\binom{\omega}{\omega}\rightarrow S_t\binom{\omega}{k}$ be defined as follows: 

\[\pi(A)(i) =\left\{ \begin{array}{rcl}
0\ \ \  & \mbox{if} & A(i)\geq t+k\\
& & \\
A(i) & \mbox{if} & 0\leq A(i)< t+k\\
\end{array}
\right. \]

Notice that $\pi$ is a surjection. For $l>k$, we will extend the function $\pi$ to $\mathcal{AR}_l$ as follows. Given $a\in\mathcal{AR}_l$ and $i$ in the domain of $a$, let

\[\pi(a)(i) =\left\{ \begin{array}{rcl}
0\ \ \  & \mbox{if} & t+k\leq a(i)< t+l\\
& & \\
a(i) & \mbox{if} & 0\leq a(i)< t+k\\
\end{array}
\right. \]

Define $s : \mathbb N\times\bigcup_k\mathcal S_t\binom{\omega}{k} \rightarrow\bigcup_{i\leq k}\mathcal{AR}_i$ as follows. Given $n<\omega$, $X\in\mathcal S_t\binom{\omega}{k}$ and any $A\in\mathcal S_t\binom{\omega}{\omega}$ such that $\pi(A)=X$, let 

\[s(n,X) =\left\{ \begin{array}{rcl}
r_n(A)\ \ \ \ \ \ \  & \mbox{if} & 0\leq n\leq k\\
\pi(r_n(A)) \ \ \ \ & \mbox{if} & n>k\\
\end{array}
\right. \]
 
With these defnitions, the structure $(\mathcal S_t\binom{\omega}{\omega}, \leq, r, (\mathcal S_t\binom{\omega}{k})_k,\circ,s)$ satisfies axioms $\textbf{A.1}-\textbf{A.7}$ and $\mathcal S_t\binom{\omega}{\omega}$ is metrically closed. Thus we get the following from Theorem \ref{main thm}. 

\begin{thm}[Carlson-Simpson \cite{CarSim}]\label{borel parameter words ramsey} For every  $Y\in\mathcal S_t\binom{\omega}{\omega}$ and every finite Borel-measurable coloring of $\mathcal S_t\binom{\omega}{k}$, $k<\omega$, there exists $X\in\mathcal S_t\binom{\omega}{\omega}|Y$ such that $\mathcal S_t\binom{\omega}{k}|X$ is monochromatic.
\end{thm}

\begin{thm}[Pr\"omel-Voigt \cite{PromVoi}]\label{parameter words ramsey} 
For every  $Y\in\mathcal S_t\binom{\omega}{\omega}$ and every finite Baire-measurable coloring of $\mathcal S_t\binom{\omega}{k}$, there exists $X\in\mathcal S_t\binom{\omega}{\omega}|Y$ such that $\mathcal S_t\binom{\omega}{k}|X$ is monochromatic.
\end{thm}

\begin{rem}
For $t=0$, Theorem \ref{borel parameter words ramsey}  is known as the Dual Ramsey Theorem. 
\end{rem}

\subsection{Ascending parameter words}\label{ascending}

Now, we will explore a special type of parameter words. Let $\mathcal S_t^{<}\binom{\beta}{\alpha}$ denote the set of all $A\in \mathcal S_t\binom{\beta}{\alpha}$ satisfying

\begin{enumerate}
\item $A^{-1}(\{t+j\})$ is finite, for  all $j<\alpha$.
\item $\max\ A^{-1}(\{t+i\})<\min\ A^{-1}(\{t+j\})$, for all $i<j<\alpha$.
\end{enumerate}

Let $\mathcal R = \mathcal S_t^{<}\binom{\omega}{\omega}$ and for every $k<\omega$ let $\mathcal R_k = \mathcal S_t^{<}\binom{\omega}{k}$. $\mathcal S_t^{<}\binom{\omega}{\omega}$ is a subset of $\mathcal S_t\binom{\omega}{\omega}$. So we can consider in this case the restictions $\leq$, $\circ$, $r$ and $s$, as defined in Section \ref{parameter}, to the corresponding domains within the context of $\mathcal S_t^{<}\binom{\omega}{\omega}$. But by letting $m=1$ in the definition of the  space $\mathcal S_t^{<}\binom{\omega,m}{\omega}$ introduced in Section \ref{general parameter}, we can easily verify that $\mathcal S_t^{<}\binom{\omega}{\omega}=\mathcal S_t^{<}\binom{\omega,1}{\omega}$. Notice that $\mathcal S_0^{<}\binom{\omega}{\omega}$ is essentially Ellentuck's space (see \cite{Ellen}) and, for all $k<\omega$, $\mathcal S_0^{<}\binom{\omega}{k}=\emptyset$. On the other hand, $\mathcal S_1^{<}\binom{\omega}{\omega}$ is Milliken's space (see \cite{Milliken}). 

\begin{thm}[Milliken, \cite{Milliken}]\label{Milliken space}
$(\mathcal S_1^{<}\binom{\omega}{\omega}, \leq, r)$ is a topological Ramsey space.
\end{thm}

Axiom {\bf A.4} for  $(\mathcal S_1^{<}\binom{\omega}{\omega}, \leq, r)$ is equivalent to Hindman's theorem \cite{Hindman}.  Again, the structure $(\mathcal S_t^{<}\binom{\omega}{\omega}, \leq, r, (\mathcal S_t^{<}\binom{\omega}{k})_k,\circ,s)$ satisfies axioms ${\bf A.1- A.7}$ and $\mathcal S_t^{<}\binom{\omega}{\omega}$ is metrically closed. Letting $m=1$ in Theorem \ref{block sequence ramsey} and Theorem \ref{general parameter space} we obtain a different proof of the following well-known results.

\begin{thm}
[Carlson, Theorem 6.9 in \cite{CarSim2}]\label{ascending space}$(\mathcal S_t^{<}\binom{\omega}{\omega}, \leq, r)$ is a topological Ramsey space.
\end{thm}

\begin{thm}[Pr\"omel-Voigt \cite{PromVoi}]\label{ascending words ramsey} For every  $Y\in\mathcal S_t^{<}\binom{\omega}{\omega}$ and every finite Baire-measurable coloring of $\mathcal S_t^{<}\binom{\omega}{k}$, there exists $A\in\mathcal S_t^{<}\binom{\omega}{\omega}|Y$ such that $\mathcal S_t^{<}\binom{\omega}{k}|A$ is monochromatic.
\end{thm}

\subsection{Partial $\mathcal G$-partitions}
 
Let $\mathcal G$ be a finite group and let $e\in \mathcal G$ denote its unit element. Also let $\nu$ be a symbol not ocurring in $\mathcal G$. Given ordinals $\alpha\leq\beta\leq\omega$, let $\mathcal S_{\mathcal G}\binom{\beta}{\alpha}$ denote the set of all mappings $A: \beta \rightarrow \{\nu\}\cup(\alpha\times \mathcal G)$ satisfying

\begin{enumerate}
\item For every $j<\alpha$ there exists $i<\beta$ such that $A(i)=(j,e)$ and $A(i')\not\in\{j\}\times \mathcal G$ for all $i'<i$.
\item $\min\ A^{-1}(\{(i,e)\})< \min\ A^{-1}(\{(j,e)\})$ for all $i<j<\alpha$.
\end{enumerate}

Elements of $\mathcal S_{\mathcal G}\binom{\beta}{\alpha}$ are known as \textit{partial} $G$-\textit{partitions of} $\beta$ \textit{into} $\alpha$ \textit{blocks}.

\medskip

For $A\in\mathcal S_{\mathcal G}\binom{\gamma}{\beta}$  and $B\in\mathcal S_{\mathcal G}\binom{\beta}{\alpha}$, the composite $A\cdot B\in\mathcal S_{\mathcal G}\binom{\gamma}{\alpha}$ is defined by

\[(A\cdot B)(i) =\left\{ \begin{array}{rcl}
\nu \ \ \ \ \ & \mbox{if} & A(i)=\nu\\
\nu \ \ \ \ \ &  \mbox{if} &  A(i)=(j,b) \mbox{ and } B(j)=\nu\\
(k, b\cdot c) & \mbox{if} & A(i)=(j,b) \mbox{ and } B(j)=(k,c)\\
\end{array}
\right. \]

Also for $(n,A)\in\mathbb N\times\mathcal S_{\mathcal G}\binom{\omega}{\omega}$ let

\[r(n,A) =\left\{ \begin{array}{rcl}
\emptyset\ \ \ \ \ \ \ \ \ \ \ \ \ \ \ \ \ \ \ \  & \mbox{if} & n= 0\\
& & \\
A\upharpoonright \min\ A^{-1}(\{(n,e)\}) & \mbox{if} & n> 0\\
\end{array}
\right. \]

Define  $A\circ B = A\cdot B$, and as before write $A\leq B$ if there exists $C$ such that $A=B\circ C$. 

\begin{thm}[Pr\"omel-Voigt \cite{PromVoi}]\label{partial G space}
$(\mathcal S_{\mathcal G}\binom{\omega}{\omega}, \leq, r)$ is a topological Ramsey space.
\end{thm}

\begin{rem}
 {\bf A.4} for the space $(\mathcal S_{\mathcal G}\binom{\omega}{\omega}, \leq, r)$ follows from case $k=1$ of Therem D in \cite{PromVoi}.
\end{rem}

Given $i<\omega$ and $A\in\mathcal S_{\mathcal G}\binom{\omega}{\omega}$, if $A(i)=(n,b)$ then we will write $A(i)_0=n$. That is, $A(i)_0$ is the first coordinate of $A(i)$. Now, let $\pi: S_t\binom{\omega}{\omega}\rightarrow S_t\binom{\omega}{k}$ be defined as follows: 

\[\pi(A)(i) =\left\{ \begin{array}{rcl}
(0,e)\ \ \  & \mbox{if} & A(i)_0\geq t+k\\
& & \\
A(i)\ \ \  & \mbox{if} & 0\leq A(i)_0< t+k\\
\end{array}
\right. \]

Notice that $\pi$ is a surjection. Again, for $l>k$, we will extend the function $\pi$ to $\mathcal{AR}_l$ as follows. For $a\in\mathcal{AR}_l$ and $i$ in the domain of $a$, if $a(i)=(n,b)$ then write $a(i)_0=n$ and let

\[\pi(a)(i) =\left\{ \begin{array}{rcl}
(0,e)\ \ \  & \mbox{if} & t+k\leq a(i)_0< t+l\\
& & \\
a(i)\ \ \ \ & \mbox{if} & 0\leq a(i)_0< t+k\\
\end{array}
\right. \]

As in the previous Section, define $s : \mathbb N\times\bigcup_k\mathcal S_{\mathcal G}\binom{\omega}{k} \rightarrow\bigcup_{i\leq k}\mathcal{AR}_i$ as follows. Given $n<\omega$, $X\in\mathcal S_{\mathcal G}\binom{\omega}{k}$ and any $A\in\mathcal S_{\mathcal G}\binom{\omega}{\omega}$ such that $\pi(A)=X$, let 

\[s(n,X) =\left\{ \begin{array}{rcl}
r_n(A)\ \ \ \ \ \ \  & \mbox{if} & 0\leq n\leq k\\
\pi(r_n(A)) \ \ \ \ & \mbox{if} & n>k\\
\end{array}
\right. \]

With these definitions the following hold:

\begin{thm}[Pr\"omel-Voigt, \cite{PromVoi}]\label{ascending words ramsey} For every  $Y\in\mathcal S_{\mathcal G}\binom{\omega}{\omega}$ and every finite Baire-measurable coloring of $\mathcal S_{\mathcal G}\binom{\omega}{k}$, there exists $X\in\mathcal S_{\mathcal G}\binom{\omega}{\omega}|Y$ such that $\mathcal S_{\mathcal G}\binom{\omega}{k}|X$ is monochromatic.
\end{thm}

\subsection{Infinite dimensional vector subspaces of $\mathbb F^{\mathbb N}$}

Given a finite field $\mathbb F$, let $\mathcal M_{\infty} = \mathcal M_{\infty}(\mathbb F)$ denote the set of all reduced echelon $\mathbb N\times\mathbb N$-matrices, $A: \mathbb N\times\mathbb N\rightarrow\mathbb F$. The $i$th column of $A$ is the function $A_n: \mathbb N\rightarrow \mathbb F$ given by $A_n(j)= A(n,j)$. For $A,B\in\mathcal M_{\infty}$, $A\leq B$ means that every column of $A$ belongs to the closure (taken in  $F^{\mathbb N}$ with the product topology) of the linear span of the columns of $B$. For $n\in\mathbb N$ and $A\in\mathcal M_{\infty}$ let $p_n(A)=\min\ \{j : A_n(j)\neq 0\}$ and define $r(0,A) = \emptyset$ and $r(n+1,A) = A\upharpoonright (n\times p_n(A))$.

\begin{thm}[Carlson \cite{Carlson}]\label{vector space} $(\mathcal M_{\infty}, \leq, r)$ is a topological Ramsey space.
\end{thm}

Now, for every postive $k\in\mathbb N$, let $\mathcal M_k = \mathcal M_k(\mathbb F)$ denote the set of all reduced echelon $\mathbb N\times k$-matrices, $A: \mathbb N\times k\rightarrow\mathbb F$. For $A\in\mathcal M_{\infty}$ and $B\in \mathcal M_{\infty}\cup\mathcal M_k$, let $A\circ B$ be the usual multiplication of matrices. Define 

\[s(n,A) =\left\{ \begin{array}{rcl}
r(n,A) \ \ \ \ \ \ \ \ \ \ \ & \mbox{if} & n\leq k\\
& & \\
 A\upharpoonright (n\times p_k(A)) & \mbox{if} & n>k\\
\end{array}
\right. \]
 
With these definitions $(\mathcal M_{\infty}, \leq, r, (\mathcal M_k)_k,\circ,s)$ satisfies {\bf A.1}--{\bf A.7}, so we get the following:

\begin{thm}[Todorcevic, \cite{Todo}]\label{matrix ramsey}
For every  $B\in\mathcal M_{\infty}$ and every finite Baire-measurable coloring of $\mathcal M_k$, there exists $A\in\mathcal M_{\infty}|B$ such that $\mathcal M_k|A$ is monochromatic.
\end{thm}

Now, let $\mathcal V_{\infty}(\mathbb F)$ denote the collection of all closed infinite-dimensional subspaces of $\mathbb F^{\mathbb N}$ and for every $k$, let $\mathcal V_k(\mathbb F)$ denote the collection of all $k$-dimensional subspaces of $\mathbb F^{\mathbb N}$. If we consider $\mathcal V_{\infty}(\mathbb F)$ and $\mathcal V_k(\mathbb F)$ with the Vietoris topology then there are natural homeomorphisms between $\mathcal V_{\infty}(\mathbb F)$ and $\mathcal M_{\infty}$, and $\mathcal V_k(\mathbb F)$ and $\mathcal M_k$. So Theorem \ref{matrix ramsey} can be restated as

\begin{cor}[Voigt \cite{Voigt}]\label{vector ramsey} For every  $W\in\mathcal V_{\infty}(\mathbb F)$ and every finite Baire-measurable coloring of $\mathcal V_k(\mathbb F)$, there exists $V\in\mathcal V_{\infty}(\mathbb F)|W$ such that $\mathcal V_k(\mathbb F)|V$ is monochromatic.
\end{cor}

\section{Final remark}

Following with the tradition started in \cite{CarSim,Todo}, the results contained in Section \ref{setting} attempt to serve as part of a unifying framework for the theory of topological Ramsey spaces. The new axioms {\bf A5}--{\bf A7}  are sufficent to capture a particular feature of a class of topological Ramsey spaces which was not revealed by the original 4 axioms proposed in \cite{Todo}. In Section \ref{setting} , we showed a characterization of those topological Ramsey spaces $\mathcal R$ with family of approximations $\mathcal{AR}= \bigcup_{k<\omega}\mathcal{AR}_k$ for which there exist topological spaces $\mathcal R_k\subseteq(\mathcal{AR}_k)^{\mathbb N}$, $k<\omega$, such that every Baire subset of $\mathcal R_k$ is Ramsey. While this characterization is essentially based on the axioms {\bf A5}--{\bf A7}, the question about the necessity of these axioms remains open. 

Finally, given $0<m<\omega$, there is a well-known understanding of the space $\FIN_m^{[\infty]}=\mathcal S_1^{<}\binom{\omega,m}{\omega}$ in terms of Functional Analysis. It has to do with the property of oscilation stabilty for Lipschitz functions defined on the sphere of the Banach space $c_0$ (see \cite{gowers}) and with the solution of the Distortion Problem (see \cite{odelslum}). This leads to the following natural question: For $t>1$, what is the interpretation (if any) of the space $\mathcal S_t^{<}\binom{\omega,m}{\omega}$ in terms Functional Analysis?

\medskip

\section*{Acknowledgement} Jos\'e Mijares would like to thank Jes\'us Nieto for valuable conversations about  the proof of Theorem \ref{pigeon hole} in Section \ref{general parameter}.

\end{document}